\documentclass{amsart}

\usepackage{amsmath}
\usepackage{amssymb}
\usepackage{amscd}
\usepackage{hyperref}

\newtheorem{theorem}{Theorem}[section]
\newtheorem{thm}[theorem]{Theorem}
\newtheorem{cor}[theorem]{Corollary}
\newtheorem{lem}[theorem]{Lemma}

\theoremstyle{definition}
\newtheorem{defn}[theorem]{Definition}

\newtheorem{ques}[theorem]{Question}

\newtheorem{exer}[theorem]{Exercise}

\theoremstyle{remark}

\newcommand{\mbb}{\mathbb}

\newcommand{\PP}{\mbb{P}}

\newcommand{\mc}{\mathcal}

\newcommand{\mcI}{\mc{I}}

\newcommand{\OO}{\mc{O}}

\newsavebox{\sembox}
\newlength{\semwidth}
\newlength{\boxwidth}

\newsavebox{\semrbox}
\newlength{\semrwidth}
\newlength{\boxrwidth}

\title
{On the space of conics on complete intersections}

\author[ Zong]{ Hong R. Zong }

\address{
Department of Mathematics\\
Princeton University \\
Princeton, NJ, 08544-1000}
\email{rzong@math.princeton.edu}

\date{\today}

\begin{document}


\begin{abstract}
We get sharp degree bound for generic smoothness and connectedness of the space of conics in low degree complete intersections which generalizes the old work about Fano scheme of lines on Hypersurfaces.

\end{abstract}


\maketitle


\section{introduction}

\begin{defn}
For a variety $X$ in $\PP^n$, we define the Fano scheme of lines (\cite{Barth}) $F(X)$ to be the Hilbert scheme of lines $Hilb_{t+1}(X)$ on $X$. And  $C(X)$ to be the Hilbert scheme of conics (degree $2$ rational curve, union of two lines, or double of one line) $Hilb_{2t+1}(X)$ on $X$.
\end{defn}

The following result is firstly proved in \cite{Barth}, which is simplified with sharp bound by Professor J. Koll\'ar in his book \cite{kollarbook}.

\begin{thm}
Let $X$ be a hyper-surface of degree $d$. Then:

 $F(X)$ is empty for general $X$ if $d > 2n-3$, smooth with dimension $2n-3-d$ for general $X$ if $d \leq 2n-3$. If $d\leq 2n-4$ (with the  only exceptional case as quadric surface), then $F(X)$ is always connected, thus is irreducible for general $X$.
\end{thm}

We generalize the above theorem to the case of lines and conics in low degree complete intersections, we note here that L. Bonavero and A. H$\ddot{o}$ring already obtained related results using cohomology computations on Grassmannians, see Section 2.3 and Proposition 3.1 in \cite{h}.

\begin{thm}\label{fanoscheme}
Let $X$ be a  complete intersection of type $d_1, \ldots, d_c$. Then:
 \begin{itemize}
 \item $F(X)$ is empty for general $X$ if $d_1+d_2+...+d_c+c > 2n-2$, smooth with dimension $2n-2-(d_1+d_2+...+d_c+c)$ for general $X$ if $d_1+d_2+...+d_c +c \leq 2n-2$. If $d_1+d_2+...+d_c +c \leq 2n-3$ (with the  only exceptional case as quadric surface), then $F(X)$ is always connected, thus is irreducible for general $X$.
 \item $C(X)$ is empty for general $X$ if $d_1+d_2+...+d_c+c/2> (3n-2)/2$, smooth with dimension $3n-1-2(d_1+d_2+...+d_c)-c$ for general $X$ if $d_1+d_2+...+d_c+c/2 \leq (3n-2)/2$,  always connected if $d_1+d_2+...+d_c+c/2 \leq (3n-3)/2$ ($n\leq 3$ case is trivial), in particular $\overline{M}_{0,0}(X,2)$ is connected.
 \end{itemize}
\end{thm}

A crucial part of the proof is some generic dimension count on conics, the rest is the same as the proof in \cite{kollarbook}.

We have an easy corollary:
\begin{cor}
On a complete intersection $X$ with type $(d_1,...,d_c)$ and $d_1+d_2+...+d_c+c/2 \leq (3n-2)/2$, all degree $2$ rational curves on $X$ can degenerated to union of two lines, or double of one line.
\end{cor}

We have another application of Theorem \ref{fanoscheme} suggested by Professor C. Voisin, for a short proof see Remark $6.4$ of \cite{TZ}:
\begin{thm}
For a Fano complete intersection $X$ with type $(d_1,...,d_c)$ with index $\geq 2$, the $1$-Griffiths group generated by algebraic $1$-cycles homologous to $0$ modulo algebraic equivalence is trivial.
\end{thm}

The readers are encouraged to find more applications of Theorem \ref{fanoscheme}. And by the same method, it should not be hard to get the sharp bound for hyper-quadrics in complete intersections:

\begin{ques}
Generalize Theorem \ref{fanoscheme} to the case of hyper-quadrics in complete intersections.
\end{ques}

\textbf{Acknowledgment:}
To Professor J\'anos Koll\'ar, for his constantly magnus support for the author. To Zhiyu Tian for raising the question and helpful discussions.

\section{preliminaries}
\subsection{Connectedness lemma}
Firstly we need a connectedness lemma for latter use.
\begin{lem}\label{connectedness2}
Let $X$, $Y$ be noetherian schemes, with proper morphism $f: X \to Y$, assume $Y$ is smooth and $\pi^1(Y)= \lbrace 1 \rbrace $. Let $S(f)\subset X$ be the locus where $f$ is not smooth. If $codim(S(f),X) \geq 2$, then $f$ has connected fibers.
\end{lem}
\begin{proof}
Consider the Stein factorization of $f$: $$X\stackrel{g}{\longrightarrow}X'\stackrel{h}{\longrightarrow} Y,$$ where $g$ is proper with connected fibers and $h$ is finite. By the assumption we know that $h$ is generically \'etale (when the characteristic is $p$, smoothness of the map at general point implies that $K(X')/K(Y)$ is a separable extension). Let $D\subset Y$ be its branch divisor, then if $D \neq 0$, $f$ will not be smooth in $f^{-1} \lbrace D\rbrace \subset X$ which is of codimension $1$, contradicting the assumption. So $D=0$, and $h$ will be \'etale outside a codimension $\geq 2$ locus $E$, hence $$h|_{X'-h^{-1}\lbrace E \rbrace }: X'-h^{-1}\lbrace E \rbrace \to Y-E$$ is a \'etale cover of $Y-E$. Now by \cite{SGA1} Corollary $X.3.3$ $$\pi^1(Y-E )=\pi^1(Y)=\lbrace 1\rbrace$$ which means $deg(h)=1$, so $f=g$ has connected fibers.

\end{proof}
\subsection{Smoothness of $C(\PP^n)$}
The following is an easy exercise, which we will need latter:
\begin{exer}\label{hilb}
$Hilb_{2t+1}(\mathbb{P}^n)$ is isomorphic to $Proj(E)$ of a rank-$6$ vector bundle $E$ over $Grass(3,n+1)$, and in particular smooth and irreducible.
\end{exer}

\subsection{Generic dimension count on conics}

\subsubsection{Smooth conic case}
\begin{lem}\label{count1}
Let $m$ be the multiplication map
$$H^0(\mathbb{P}^1, \OO(1))\times (\bigoplus_{i=1,...,c} H^0(\PP^1,\OO(d_i-1))) \to \bigoplus_{i=1,...,c} H^0(\PP^1,\OO(d_i) ). $$ For codimension $1$ linear subspace $$V \subset \bigoplus_{i=1,...,c} H^0(\PP^1,\OO(d_i) ),$$ let $$m^{-1}(V) :=\{ f \in \bigoplus_{i=1,...,c} H^0(\PP^1,\OO(d_i-1)), f \cdot H^0(\PP^1,\OO(1)))\subset V\}.$$
Let $(f_{i,1},...,f_{i,c}),i=1,...,dim\ m^{-1}(V)$ be a basis of $m^{-1}(V)$.

Then either of the followings is true:
\begin{itemize}
\item There is a point $p\in \mathbb{P}^1$ such that $$rank\{f_{i,j}(p)\}<c$$
\item $m^{-1}(V)$ is of codimension $2$ in $\bigoplus_{i=1,...,c} H^0(\PP^1,\OO(d_i-1))$.
\end{itemize}
\end{lem}

\begin{proof}
First we note that $H^0(\mathbb{P}^1, \OO(d))$ can be identified with all polynomials spanned by $\{1,x,x^2,...,x^d\}$.
Let $x_1,...,x_c$ be generators for each $H^0(\mathbb{P}^1,\OO(d_i))$, and $H^0(\mathbb{P}^1,\OO(1))$ spanned by $\{1,x\}$, then $m$ is simply defined by: $$m(x,x_1^{i_1},...,x_c^{i_c})=(x_1^{i_1+1},...,x_c^{i_c+1}).$$ Denote the coordinate of $x_i^{j}$ to be $x_{i,j}$, and
let $V$ be defined by equation $$\sum_{1\leq i \leq d_1}a_{1,i}x_{1,i}+\sum_{1\leq i \leq d_2}a_{2,i}x_{2,i}...+\sum_{1\leq i \leq d_c}a_{c,i}x_{c,i}=0,$$ then $m^{-1}(V)$ is defined by
$$\sum_{1\leq i \leq d_1-1}a_{1,i}x_{1,i}+\sum_{1\leq i \leq d_2-1}a_{2,i}x_{2,i}...+\sum_{1\leq i \leq d_c-1}a_{c,i}x_{c,i}=0$$ and $$\sum_{1\leq i \leq d_1-1}a_{1,i}x_{1,i+1}+\sum_{1\leq i \leq d_2-1}a_{2,i}x_{2,i+1}...+\sum_{1\leq i \leq d_c-1}a_{c,i} x_{c,i+1}=0.$$
So it will be of codimension $2$ unless there is $p=[s,t]\in \mathbb{P}^1$ such that $$s\cdot a_{i,j}=t\cdot a_{i,j+1}.$$ We may assume that $p=[0,1]$, then $m^{-1}(V)$ is defined by $$a_{1,d_1}x_{1,d_1-1}+...a_{c,d_c}x_{c,d_c-1}=0.$$ And the matrix $\{f_{i,j}(p)\}$ will has elements as $$f_{i,j}(p)=x_{j,d_1-1}(f_i)$$--the defining equation of $m^{-1}(V)$ simply implies that this $dim\ m^{-1}(V) \times c$ matrix satisfies a non-trivial linear relation in its $j$-index, so its rank will be strictly less than $c$.

\end{proof}

\begin{lem}\label{count2}
let $m_2$(resp\-.$m_4$) be the multiplication map
$$H^0(\mathbb{P}^1, \OO(2))\times (\bigoplus_{i=1,...,c} H^0(\PP^1,\OO(2d_i-2)))\to \bigoplus_{i=1,...,c} H^0(\PP^1,\OO(2d_i) )  $$ resp.\-$$H^0(\mathbb{P}^1, \OO(4))\times (\bigoplus_{i=1,...,c} H^0(\PP^1,\OO(2d_i-4))) \to \bigoplus_{i=1,...,c} H^0(\PP^1,\OO(2d_i).$$
For codimension $1$ linear subspace $$V \subset \bigoplus_{i=1,...,c} H^0(\PP^1,\OO(2d_i) ),$$ let $$m_2^{-1}(V) :=\{ f \in \bigoplus_{i=1,...,c} H^0(\PP^1,\OO(2d_i-2)), f \cdot H^0(\PP^1,\OO(2)))\subset V\}$$ resp.\-$$m_4^{-1}(V) :=\{ f \in \bigoplus_{i=1,...,c} H^0(\PP^1,\OO(2d_i-4)), f \cdot H^0(\PP^1,\OO(4)))\subset V\}.$$
Let $(f_{i,1},...,f_{i,c}),i=1,...,dim\ m_2^{-1}(V)$(resp.$dim\ m_4^{-1}(V)$) be a basis of $m_2^{-1}(V)$(resp.\-$m_4^{-1}(V)$).

Then either of the follows is true:
\begin{itemize}
\item There are $2$ (resp.\-$4$) points $p_1,p_2 \in \mathbb{P}^1$ (resp. $p_1,p_2,p_3,p_4 \in \mathbb{P}^1$) such that $$rank\{f_{i,j}(p_i)\}<c$$ ($p_i$'s might coincide with each other, resulting in even lower codimension of $m_{*}^{-1}(V)$).
\item $m_2^{-1}(V)$ (resp. $m_4^{-1}(V)$is of codimension $3$(resp.\-$5$) in $\bigoplus_{i=1,...c} H^0(\PP^1,\OO(2d_i-2))$(resp.\-$\bigoplus_{i=1,...,c} H^0(\PP^1,\OO(2d_i-4))$).
\end{itemize}
\end{lem}
\begin{proof}
We only prove the $m_2$ part here: $H^0(\mathbb{P}^1, \OO(2d))$ can be identified with all polynomials spanned by $\{1,x,x^2,...,x^{2d}\}$.
Let $x_1,...,x_c$ be generators for each $H^0(\mathbb{P}^1,\OO(2d_i))$, and $H^0(\mathbb{P}^1,\OO(2))$ spanned by $\{1,x,x^2\}$.

Then $m_2$ is simply defined by: $$m_2(x^i,x_1^{i_1},...,x_c^{i_c})=(x_1^{i_1+i},...,x_c^{i_c+i}).$$ Denote the coordinate of $x_i^{j}$ as $x_{i,j}$, and
let $V$ be defined by equation: $$\sum_{1\leq i \leq 2d_1}a_{1,i}x_{1,i}+\sum_{1\leq i \leq 2d_2}a_{2,i}x_{2,i}...+\sum_{1\leq i \leq 2d_c}a_{c,i}x_{c,i}=0.$$
Then $m_2^{-1}(V)$ is defined by:
$$\sum_{1\leq i \leq 2d_1-2}a_{1,i}x_{1,i}+\sum_{1\leq i \leq 2d_2-2}a_{2,i}x_{2,i}...+\sum_{1\leq i \leq 2d_c-2}a_{c,i}x_{c,i}=0$$  $$\sum_{1\leq i \leq 2d_1-2}a_{1,i}x_{1,i+1}+\sum_{1\leq i \leq 2d_2-2}a_{2,i}x_{2,i+1}...+\sum_{1\leq i \leq 2d_c-2}a_{c,i} x_{c,i+1}=0$$ and $$\sum_{1\leq i \leq 2d_1-2}a_{1,i}x_{1,i+2}+\sum_{1\leq i \leq 2d_2-2}a_{2,i}x_{2,i+2}...+\sum_{1\leq i \leq 2d_c-2}a_{c,i} x_{c,i+2}=0.$$
So it will be of codimension $3$ unless there is $p_1=[s_1,t_1], p_2=[s_2,t_2]\in \mathbb{P}^1$ such that $$a_{i,j+2}=(s_1/t_1+s_2/t_2)\cdot a_{i,j+1}+s_1/t_1 \cdot s_2/t_2\cdot a_{i,j}$$ then everything would be clear after solving this recurrence relation.

\end{proof}

\subsubsection{Union of two lines case}
\begin{lem}\label{count3}
let $D$ be union of two $\mathbb{P}^1$'s glued at a point with nodal singularity, $\OO(d)$ be the line bundle gluing from $\OO(d)$'s of two pieces.

Let ${m'}_1$ (resp.\- ${m'}_2$) be the multiplication map
$$H^0(D, \OO(1))\times (\bigoplus_{i=1,...,c} H^0(D,\OO(d_i-1)))\to \bigoplus_{i=1,...,c} H^0(D,\OO(d_i))  $$ resp.\-$$H^0(D, \OO(2))\times (\bigoplus_{i=1,...,c} H^0(D,\OO(d_i-2)))\to \bigoplus_{i=1,...,c} H^0(D,\OO(d_i)).$$
 For codimension $1$ linear subspace $$V \subset \bigoplus_{i=1,...,c} H^0(D,\OO(d_i) ),$$ let $${m'}_1^{-1}(V) :=\{ f \in \bigoplus_{i=1,...,c} H^0(D,\OO(d_i-1)), f \cdot H^0(D,\OO(1)))\subset V \}$$ resp.\-$${m'}_2^{-1}(V) :=\{ f \in \bigoplus_{i=1,...,c} H^0(D,\OO(d_i-2)), f \cdot H^0(D,\OO(2)))\subset V \}.$$
Let $(f_{i,1},...,f_{i,c}),i=1,...,dim\ {m'}_1^{-1}(V)$ (resp.\-$dim\ {m'}_2^{-1}(V)$) be a basis of ${m'}_1^{-1}(V)$(resp.\-${m'}_2^{-1}(V)$).

Then either of the followings is true:
\begin{itemize}
\item There are $2$(resp.\-$4$) points $p_1,p_2 \in D$ (resp.\- $p_1,p_2,p_3,p_4 \in D$ )such that $$rank\{f_{i,j}(p_i)\}<c$$.($p_i$'s might coincide with each other, resulting in even lower codimension of ${m'}_{*}^{-1}(V)$).
\item ${m'}_1^{-1}(V)$ (resp.\-${m'}_2^{-1}(V)$) is of codimension $3$(resp.\-$5$) in $\bigoplus_{i=1,...,c} H^0(D,\OO(d_i-1))$(resp.\-$\bigoplus_{i=1,...,c} H^0(D,\OO(d_i-2))$).
\end{itemize}
\end{lem}
\begin{proof}
The sections here come from glueing each section from two lines, and the argument in Lemma \ref{count1} applies.
\end{proof}

\subsubsection{Double of a line case}
\begin{lem}\label{count4}
let $L$ be double $\mathbb{P}^1$ embedded in $\mathbb{P}^2$ as double of a line, $\OO(d)$ be the line bundle restricting from $\OO(d)$ of $\mathbb{P}^2$.

Let ${m''}_1$ (resp.\- ${m''}_2$) be the multiplication map
$$H^0(L, \OO(1))\times (\bigoplus_{i=1,...,c} H^0(L,\OO(d_i-1))) \to \bigoplus_{i=1,...,c} H^0(L,\OO(d_i) )$$ resp.\-$$ H^0(L, \OO(2))\times (\bigoplus_{i=1,...,c} H^0(L,\OO(d_i-2)))\to \bigoplus_{i=1,...,c}H^0(L,\OO(d_i)).$$
For codimension $1$ linear subspace $$V \subset \bigoplus_{i=1,...,c} H^0(L,\OO(d_i) ),$$ let $${m''}_1^{-1}(V) :=\{ f \in \bigoplus_{i=1,...,c} H^0(L,\OO(d_i-1)), f \cdot H^0(L,\OO(1)))\subset V \}$$ resp.\-$${m''}_2^{-1}(V) :=\{ f \in \bigoplus_{i=1,...,c} H^0(L,\OO(d_i-2)), f \cdot H^0(L,\OO(2)))\subset V \}.$$
Let $(f_{i,1},...,f_{i,c}),i=1,...,dim\ {m''}_1^{-1}(V)$ (resp.\-$dim\ {m''}_2^{-1}(V)$) be a basis of ${m''}_1^{-1}(V)$(resp.\-${m''}_2^{-1}(V)$).

Then either of the follows is true:
\begin{itemize}
\item There are $1$(resp.\-$2$) point $p_1 \in \mathbb{P}^1$ (resp.\- $p_1,p_2 \in \mathbb{P}^1$ )such that $$rank\{f_{i,j}(p_i)\}<c$$.($p_i$'s might coincide with each other, resulting in even lower codimension of ${m''}_{*}^{-1}(V)$).
\item ${m''}_1^{-1}(V)$ (resp.\-${m''}_2^{-1}(V)$) is of codimension $3$(resp.\-$5$) in $\bigoplus_{i=1,...,c} H^0(L,\OO(d_i-1))$(resp.\-$\bigoplus_{i=1,...,c} H^0(L,\OO(d_i-2)) $).
\end{itemize}
\end{lem}
\begin{proof}
In this case we have $H^0(L,\OO(d))\cong H^0(\mathbb{P}^1, \OO(d))\oplus t\cdot H^0(\mathbb{P}^1,\OO(d-1))$ where $t$ is a formal variable with $t^2=0$, now the multiplication map $H^0(L,\OO(d_1))\times H^0(L,\OO(d_2) \to H^0(L,\OO(d_1+d_2))$ is simply sending $(s_1+ts_2, s_1'+ts_2')$ to $s_1s_1'+t(s_2s_1'+s_1s_2')$ and the rest follows from the argument in Lemma \ref{count1}.
\end{proof}

\section{proof}
\begin{thm}\label{fanoscheme1}
Let $X$ be a  complete intersection of type $d_1, \ldots, d_c$. Then:

 \begin{itemize}
 \item $F(X)$ is empty for general $X$ if $d_1+d_2+...+d_c+c > 2n-2$, smooth with dimension $2n-2-(d_1+d_2+...+d_c+c)$ for general $X$ if $d_1+d_2+...+d_c +c \leq 2n-2$. If $d_1+d_2+...+d_c +c \leq 2n-3$ (with the  only exceptional case as quadric surface), then $F(X)$ is always connected, thus is irreducible for general $X$.
 \item $C(X)$ is empty for general $X$ if $d_1+d_2+...+d_c+c/2> (3n-2)/2$, smooth with dimension $3n-1-2(d_1+d_2+...+d_c)-c$ for general $X$ if $d_1+d_2+...+d_c+c/2 \leq (3n-2)/2$,  always connected if $d_1+d_2+...+d_c+c/2 \leq (3n-3)/2$ ($n\leq 3$ case is trivial), in particular $\overline{M}_{0,0}(X,2)$ is connected.
 \end{itemize}
\end{thm}
\begin{proof}
For line case, the proof is almost word to word translation of the arguments in Section $V.4$ of \cite{kollarbook}, with the Lemma $4.3.11$ used there replaced by Lemma \ref{count1} above.

For conic case, consider the incidence correspondence $$\mcI \subset Hilb_{2t+1}(\mathbb{P}^n) \times H:=H_{n,d_1} \times H_{n,d_2} \times ...\times H_{n,d_c}$$ where $H$ parametrizes all complete intersections of degree $(d_1,...,d_c)$ and $\mcI$ is defined by all hyper-surfaces defining the complete intersections that vanish at the conic $[C] \in Hilb_{2t+1}$.

We have two projections $$\mcI \stackrel{p_1}{\longrightarrow}  Hilb_{2t+1}(\mathbb{P}^n), \mcI \stackrel{p_2}{\longrightarrow}  H.$$
One can see that $p_1$ has smooth fibers as product of linear spaces in $H_{n,d_i}$'s. So by smoothness of $Hilb_{2t+1}(\mathbb{P}^n)$, we have that $\mcI$ is smooth.

Now one has to show that $p_2$ has connected fibers:
For any $C\subset X$ in $\mcI$ let $C$ be defined by $$x_0=...=x_{n-3}=L_2(x_{n-2},x_{n-1},x_{n})=0$$ where $L_2$ is a possibly reducible degree $2$ polynomial in the plane $x_0=...=x_{n-3}=0$ and
let $X$ be defined by $f_1=...=f_c=0$. Then $C\subset X$ implies that $$f_i=f_{i,0}x_0+f_{i,1}x_1+...+f_{i,n-3}x_{n-3}+f_{i,n-2}L_2.$$
Let $Z(\mcI)\subset \mcI$ be the set where $rank\{f_{i,j}\}(p)<c$ for some point $p\in C$, where $j=0,1,...,n-3$. It is easy to see that these two conditions are independent of the $x_0,...,x_{n-3}, L_2$ we choose, and $$codim(Z(\mcI),\mcI) \geq 2$$ if $n\geq 4$ and  $d_1+d_2+...+d_c+c/2\leq (3n-1)/2$.($n\leq 3$ case is trivial).

Let $S(\mcI) \subset \mcI$ be the locus where $p_2$ is not smooth along $C$.

In the locus $\mcI-Z(\mcI)$ where $X$ is smooth at all points of $C$, consider the normal bundle sequence:

$$0 \to N_{C|X} \to \OO_{C}(1)^{n-2} \oplus \OO_{C}(2) \stackrel{ \{ f_{i,j} \} }{\longrightarrow}\bigoplus_{i=1,...,c} \OO_{C}(d_i) \to 0$$

Then $Z(\mcI)-S(\mcI)$ is contained in the locus where $$H^1(C,N_{C|X})=0$$ which is equivalent to $$H^0(C,\OO(d_i))=\sum_{j}f_{i,j}H^0(C,\OO(1))+f_{i,n-2}H^0(C,\OO(2)).$$

So $S(\mcI)$ is cut out by $$\cup_{V} \{C \subset X |\sum_{i,j}f_{i,j}H^0(C,\OO(1))+\sum_{i} f_{i,n-2}H^0(C,\OO(2)) \subset V\}$$ where $V$ varies through all possible codimension $1$ linear subspaces of  $$\bigoplus_{i=1,...,c} H^0(C,\OO(d_i)).$$ Now Lemma \ref{count2}, Lemma \ref{count3} and Lemma \ref{count4} show that $$codim(S(\mcI),X - Z(\mcI))\geq 3n-1-2(d_1+...+d_c)-c.$$
And the general fiber of $p_2$ has dimension $3n-2-2(d_1+...+d_c)-c$, so the general smoothness follows.

When $$3n-1-2(d_1+...+d_c)-c\geq 2,$$ namely $d_1+...+d_c+c \leq (3n-3)/2$, we have $$codim( S(\mcI), \mcI-Z(\mcI)) \geq 3n-1-2(d_1+...+d_c)-c \geq 2$$ and $$codim(Z(\mcI),\mcI) \geq 2,$$
which implies the connectedness of fibers of $p_2$ by Lemma \ref{connectedness2}.

\end{proof}


\end{document}